\documentclass{amsart}

\usepackage[utf8]{inputenc}

\usepackage{amscd,graphicx,psfrag}
\usepackage{amssymb}

\usepackage{mathrsfs}
\usepackage{color}
\usepackage{tikz-cd}

\usepackage{hyperref}

\newcommand{\wt}{\widetilde}
\newcommand{\sm}{\smallsetminus}

\renewcommand{\th}{\theta}

\renewcommand{\phi}{\varphi}
\newcommand{\si}{\sigma}

\newcommand{\De}{\Delta}
\newcommand{\Si}{\Sigma}

\newcommand{\RR}{{\mathbb R}}

\newcommand{\Diff}{\operatorname{Diff}}

\renewcommand{\int}{\operatorname{int}}

\newcommand{\ConG}{\mathscr{C}}
\newcommand{\vConG}{\mathscr{V}\!\mathscr{C}}

\newtheorem{theorem}{Theorem}[section]
\newtheorem*{thmintro}{Theorem}
\newtheorem{lemma}[theorem]{Lemma}

\newtheorem{corollary}[theorem]{Corollary}

\theoremstyle{definition}     
\newtheorem{definition}[theorem]{Definition}

\theoremstyle{remark}
\newtheorem{remark}[theorem]{Remark}
\newtheorem{example}[theorem]{Example}

\newtheorem*{ackn}{Acknowledgements}
\newtheorem*{conv}{Conventions}

\begin{document}
\title[Concordance group of virtual knots]{Concordance group of virtual knots}

\author{Hans U. Boden}
\address{Mathematics \& Statistics, McMaster University, Hamilton, Ontario}
\email{boden@mcmaster.ca}

\author{Matthias Nagel}
\address{Université du Québec à Montréal, Montréal, Canada}
\email{nagel@cirget.ca}
\urladdr{http://thales.math.uqam.ca/~matnagel/}
 
\subjclass[2010]{Primary: 57M25, Secondary: 57M27}
\keywords{Concordance, slice knot, virtual knot.}
\date{\today}

\begin{abstract}
We study concordance of virtual knots. Our main result is that a classical knot~$K$ is virtually slice
if and only if it is classically slice. From this we deduce that 
the concordance group of classical knots embeds into
the concordance group of long virtual knots.
\end{abstract} 

\maketitle

\section{Introduction}
Virtual knot theory, discovered by Kauffman~\cite{Ka99}, is a nontrivial extension of
classical knot theory. Indeed,  
Goussarov, Polyak, and Viro proved that any two
classical knots are equivalent as virtual knots if and only if
they are equivalent as classical knots~\cite[Theorem 1.B]{GPV00}. 
Their result served to motivate many subsequent developments, because it predicted that many
classical knot and link invariants can be extended to virtual knots and links.

This result from~\cite{GPV00} was originally deduced from the classical Waldhausen's
theorem~\cite[Corollary 6.5]{Wa68},
but it can also be derived from Kuperberg's theorem~\cite{Ku03}. 
In the latter formulation, one represents virtual knots
geometrically as knots in thickened surfaces up to stable equivalence, and
Kuperberg's theorem tells us that the minimal genus representative is unique up to diffeomorphism.

Concordance of virtual knots has recently become an area of active interest, and many basic questions are still open.
One important question, which was raised both by Turaev~\cite[Section 2.2]{Tu08} and
by Kauffman~\cite[p.~336]{Ka15}, is the following:
\emph{If two classical knots are concordant as virtual ones, are they concordant in the usual sense?}
Our main result gives an affirmative answer to this question.
\begin{thmintro}
If two classical knots are concordant as virtual knots, then they are also concordant 
as classical knots.
\end{thmintro}
 
This result can be viewed as the analogue in concordance of the earlier
result of Goussarov, Polyak, and Viro~\cite{GPV00}, and consequently we hope
that it will stimulate further research on the problem of extending concordance
invariants from the classical to the virtual setting.   
In fact, there are already exciting new developments along these lines, for
instance the extension of the Rasmussen $s$-invariant to virtual knots given by
Dye, Kaestner, and Kauffman~\cite{DKK14}. 

We give a brief overview of the rest of the paper.
In Section~\ref{section2}, we introduce virtual knots as knots in thickened
surfaces up to stable equivalence. We recall Turaev's definition of virtual
knot concordance in Subsection~\ref{subsection2-2}, and we state and prove our
main result in Subsection~\ref{subsection2-3}. In Section \ref{section3}, we
introduce long virtual knots and construct the \emph{virtual knot concordance
group} $\vConG.$ We show that a long virtual knot $K$ is virtually slice if and
only if its closure $\overline{K}$ is, and we use it to deduce injectivity of
the natural homomorphism $\psi \colon \ConG \to \vConG$ from the classical
concordance group to the virtual concordance group.

\begin{conv}
All manifolds are assumed smooth and all knots are assumed oriented. Throughout the paper,  we work with smooth concordance.
\end{conv}


\section{Virtual knots and concordance} \label{section2}
In this section, we introduce stable equivalence of knots in thickened surfaces
and use them to define virtual knots. This gives rise to a natural notion of
concordance for virtual knots, which allows for a bordism between the two
surfaces whose thickenings contain representatives of the two virtual knots,
and requires also an embedded annulus cobounding the two knots.
 
\subsection{Diagrams and stable equivalence} 
It will be convenient for us to regard virtual knots geometrically as knots in
thickened surfaces, and we take a moment to explain this point of view.
\begin{definition}
A \emph{thickened surface}~$\Si \times I$ is a product of a closed, connected,  oriented surface~$\Si$ with
the interval~$I=[-1,1]$.
A knot~$K$ in a thickened surface~$\Si \times I$ is a $1$-dimensional submanifold~$K$
in the interior of $\Si \times I$ which is diffeomorphic to a circle. 
\end{definition}
Just as classical knots in $S^3$ are considered up to ambient isotopy, we consider knots in thickened surfaces up 
to stable equivalence~\cite{CKS02}. We take a moment to recall this carefully. 

\begin{definition}\label{def:StableEquiv}
\emph{Stable equivalence} on knots in thickened surfaces is generated by the
following operations, which transform a given knot $K$ in a thickened surface
$\Si \times I$ into a new knot $K'$ in a possibly different thickened surface
$\Si' \times I$.
\begin{enumerate}
\item\label{DiffEquiv}  
Let $f\colon \Si \times I \to \Si' \times I$ be an orientation-preserving diffeomorphism
sending the orientation class of $\Si$
to that of $\Si'$. (Notice that this implies that $f(\Si \times \{1\}) = \Si' \times \{1\}$ and $f(\Si \times \{-1\}) = \Si' \times \{-1\}$.) The knot $K'=f(K)$ in~$\Si' \times I$ is said to be obtained from $K$ in $\Si \times I$ by a \emph{diffeomorphism}.
\item  Let $h \colon S^0 \times D^2 \to \Si$ be the attaching region for a
$1$-handle that is disjoint from the image of $K$ under projection $\Si \times
I \to \Si$, then  $0$-surgery  on $\Si$ along $h$ is the surface 
\[ \Si' := \Si \sm h(S^0 \times D^2) \cup_{S^0 \times S^1} D^1 \times S^1. \] 
The knot~$K'$ is the image of the knot $K$ in $\Si' \times I$, and we say that 
it is the knot obtained from $K$ by \emph{stabilisation}.
\item \emph{Destabilisation} is the inverse operation, and it involves cutting
$\Si \times I$ along a \emph{vertical annulus}~$A$ and attaching two copies of $D^2 \times I$ along the two annuli. 
If the resulting thickened surface is disconnected, then we keep only the component containing $K$.   
\end{enumerate}
Note that in (3), an annulus $A$ in $\Si \times I$ is called \emph{vertical} if there is an embedded circle~$\gamma \subset \Si$ such that $A=\gamma \times I \subset \Si \times I$.
An equivalence class under the equivalence relation generated by (1), (2), and (3) above is called a \emph{virtual knot}. 
\end{definition}

Virtual links admit a similar description as links in $\Si \times I$, though $\Si$ need not be connected. We abuse notation slightly and use $K$ for the virtual knot, so $K$ refers to an equivalence class of knots in thickened surfaces.

Given a virtual knot $K$, then any knot  in its equivalence class will be
called a \emph{representative} for $K$. A representative is therefore a knot in
a thickened surface $\Si \times I$. 
\begin{definition}
The \emph{virtual genus} of a virtual knot~$K$ is the minimum
\[ vg(K) := \min \{g(\Si)\mid \text{$\Si \times I$  contains a representative for $K$}\}, \]
where $g(\Si)$ denotes the genus of the surface $\Si$.
\end{definition} 

A classical knot~$K \subset S^3$ can be isotoped to be disjoint from the two points~$\{0, \infty\}$.
Thus, we can view it as a knot  in the thickened surface~$S^2 \times I$. The associated virtual knot
is independent of the choice of isotopy, and we call such a knot \emph{classical}.
Therefore a virtual knot is classical if and only if its virtual genus is zero.

Kuperberg~\cite[Theorem 1]{Ku03} proved a strong uniqueness result for
minimal genus representatives.  Namely, he showed that if $(K,\Si \times I)$ and 
$(K',\Si' \times I)$ are two minimal genus representatives for the same virtual knot,
then $K'=f(K)$ for some diffeomorphism $f\colon \Si \times I \to \Si' \times I$ as in (\ref{DiffEquiv}) of Definition~\ref{def:StableEquiv} above.

For the sake of completeness, we relate the geometric definition of virtual knots to the usual 
diagrammatic definition.

A \emph{virtual knot diagram} is a regular immersion of the circle $S^1$ in the plane
$\RR^2$ with double points that are either classical or virtual. Real
crossings are drawn with one arc over the other whereas virtual crossings are
drawn with circles around them.

Two virtual knot diagrams are equivalent if they are related by planar
isotopies and \emph{generalised Reidemeister moves}. These consist of the three
usual Reidemeister moves together with three additional moves just like the
usual Reidemeister moves but with only virtual crossings, and one more move
called the mixed move which is depicted in Figure~\ref{MM}. A virtual knot is
defined to be an equivalence class of virtual knot diagrams, and virtual links
are defined similarly as  equivalence classes of virtual link diagrams.

\begin{figure}[ht]
\centering
\includegraphics[scale=0.65]{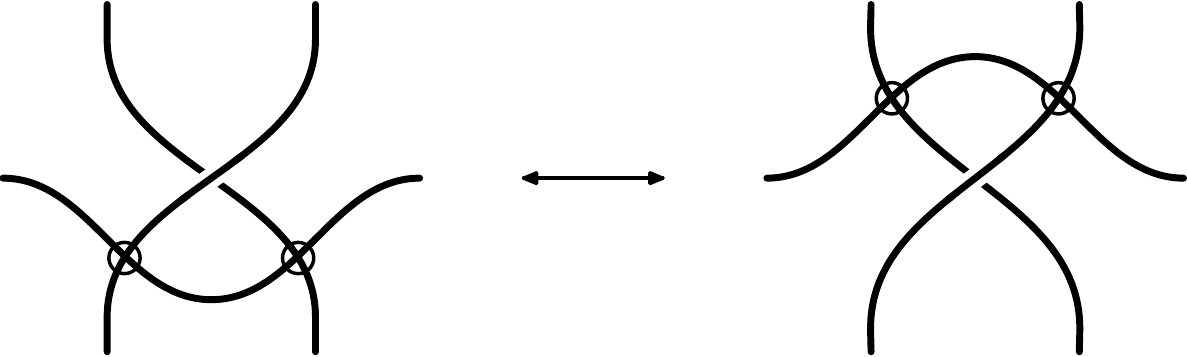}
\caption{The mixed move.}
\label{MM}
\end{figure}      

Given a virtual knot diagram~$D$ of a virtual knot~$K$, there is a canonical surface~$\Si$ called the
Carter surface constructed from $D$ as follows~\cite{KK00}:  attach two
intersecting bands at every classical crossing and two non-intersecting bands
at every virtual crossing as in Figure~\ref{band-surface}. Attaching
non-intersecting and non-twisted bands along the remaining arcs of $D$, and
filling in all boundary components with 2-disks, we obtain a closed oriented
surface~$\Si$ whose thickening~$\Si \times I$ contains a representative of $K$.
Conversely, let $K$ be a knot in a thickened surface $\Si \times I$ and $U \subset \Si$ 
a neighbourhood of the image of $K$ under projection $\Si \times I \to \Si$. 
If $f\colon U \to \RR^2$ is an orientation preserving immersion,
then the image of $K$ under $f$ is a virtual knot diagram $D$ whose classical
crossings correspond to those of $K$ and whose virtual crossings, which are the
rest of them, are the result of the immersion $f$. This virtual knot diagram
$D$ depends on the choice of immersion $f$, but any two such diagrams are
equivalent via detour moves~\cite{Ka15}.

This establishes a one-to-one correspondence between virtual knot
diagrams modulo the generalised Reidemeister moves and knots in thickened
surfaces up to stable equivalence~\cite{CKS02}.

\begin{figure}[h]
\centering
\includegraphics[scale=0.65]{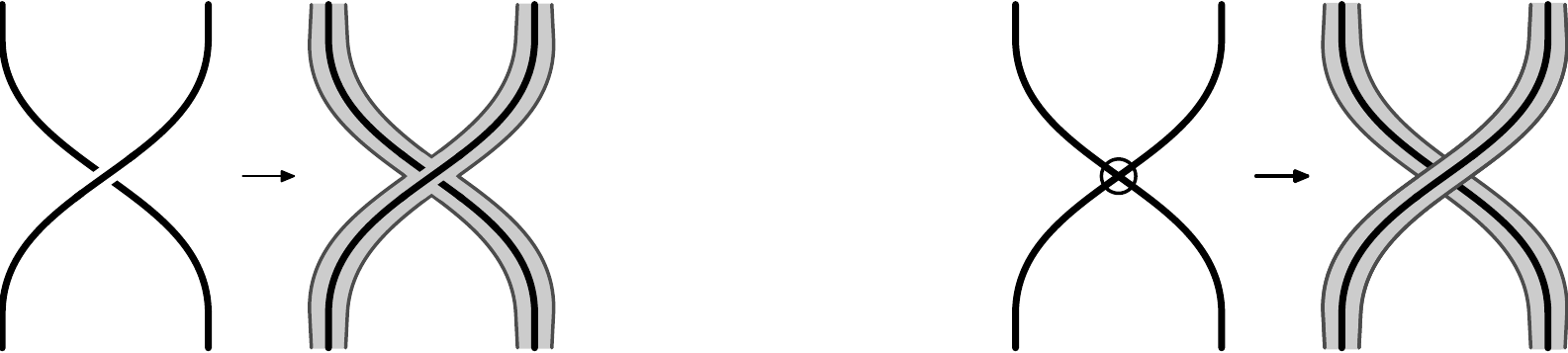}
\caption{The bands for classical and virtual crossings.}
\label{band-surface}
\end{figure}


\subsection{Virtual concordance}  \label{subsection2-2}
In this section, we define concordance and sliceness for virtual knots in terms
of their representative knots in thickened surfaces.  
We follow Turaev~\cite[Section 2.1]{Tu08} in defining virtual knot concordance.
 
If $K$ is an oriented knot in a thickened surface $\Si \times I$, its reverse is the 
knot $K^r$  obtained by changing the orientation of $K$, and its mirror image is the knot $K^m$ obtained by changing the orientation of $\Si \times I$.
These operations commute with one another, and we use
$-K=K^{rm}$ to denote the knot obtained by taking the
mirror image of the reverse knot.
\begin{definition}\label{defn:VirtualSlice}
\begin{enumerate}
\item Two given knots $K_0 \subset \Si_0 \times I$ and $K_1 \subset \Si_1 \times I$ 
in thickened surfaces are \emph{virtually concordant} if there exists a connected  
oriented $3$-manifold~$W$ with  
$\partial W \cong -\Si_0 \sqcup \Si_1$ and an annulus~$A \subset W \times I$
cobounding $-K_0$ and $K_1$.

\item A knot $K \subset \Si \times I$ is called \emph{virtually slice} if it is
concordant to the unknot. Equivalently, the knot $K$ is virtually slice if
there exists a connected $3$-manifold~$W$ with   
$\partial W \cong \Si$ and a 2-disk~$\De \subset W \times I$ cobounding $K$. We call $\De$
a \emph{slice disk} for $K$.
\end{enumerate}
\end{definition}

Clearly, virtual concordance is an equivalence relation on knots in thickened
surfaces. For instance, transitivity follows by stacking the two concordances
in the usual way. The next result shows that two stably equivalent knots are
virtually concordant to one another.  Thus, it follows that virtual
concordance defines an equivalence relation on virtual knots. 

\begin{lemma}\label{lem:ConcordanceWellDefined}
Suppose $K_0 \subset \Si_0\times I$ and $K_1 \subset \Si_1\times I$ represent the same virtual knot.
Then $K_0$ and $K_1$ are virtually concordant.
\end{lemma}
\begin{proof}
It is enough to find a 3-manifold $W$ and an annulus $A \subset W \times I$ realising a concordance between 
knots in surfaces transformed into each other by one of the operations generating
stable equivalence, see Definition~\ref{def:StableEquiv}.

One can verify that this is possible in each case.
\end{proof}

Kauffman~\cite{Ka15} re-expressed concordance purely in terms of virtual knot diagrams. A
\emph{concordance} between two virtual knot diagrams $K_0$ and $K_1$ consists of
a series of generalised Reidemeister moves together with a collection of saddle
moves, births, and deaths that transform $K_0$ into $K_1$.
As usual, one requires
that the total number of births and deaths equals the number of saddle moves, see \cite[Section 3]{Ka15}. 
This condition is equivalent to the requirement that the knots cobound an annulus. 

\begin{example} 
To illustrate this, we recall from \cite{Ka15} the argument that the Kishino knot $K$ is virtually slice. To see this, perform a saddle move along the dotted line on the left of Figure~\ref{Kishino-slice}. The result is a virtual link diagram  on the right,  
which is easily seen to be equivalent to the unlink with two components.
Filling them in with 2-disks gives a slice disk for $K$, showing that the Kishino knot is virtually slice.
\begin{figure}[ht]
\centering
\includegraphics[scale=1.60]{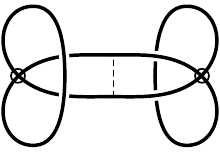} \qquad \qquad \qquad \includegraphics[scale=1.60]{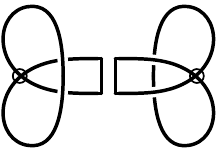}
\caption{Slicing the Kishino knot.}
\label{Kishino-slice}
\end{figure}
\end{example}

Although it is not immediately obvious, Kauffman's diagramatic notion of
virtual concordance is equivalent to Definition~\ref{defn:VirtualSlice}.
Indeed, the equivalence of the two definitions of virtual knot concordance had
been established previously by Carter, Kamada, and Saito \cite[Lemma 12]{CKS02}. 
We also refer to~\cite[Section 1.2]{CK15} for further discussion
on this point, and we thank Micah Chrisman for sharing this observation.

\subsection{Virtual concordance of classical knots} \label{subsection2-3}

Suppose $K \subset S^3$ is a classical knot and suppose that $K$ is slice. As
explained earlier, we can view $K$ as a virtual knot by arranging that $K$ lies
in a neighbourhood of the standard sphere $S^2 \times I \subset S^3$. If $\De \subset B^4$ is a slice disk for $K$,
then we can further arrange that $\De$ lies in $(S^2 \times I) \times I \subset
B^4$. Thus, $K$ is seen to be virtually slice in the sense of
Definition~\ref{defn:VirtualSlice}. 

The next theorem is our main result, and it gives a characterisation of virtual sliceness for classical knots.

\begin{theorem}\label{thm:ClassicVSlice}
A classical knot is virtually slice if and only if it is slice.
\end{theorem}

An immediate consequence of this theorem is that there are infinitely many
distinct virtual concordance classes of virtual knots. This fact had been
noted by Turaev using his polynomial invariants $u_\pm(K)$ 
\cite[Theorems 1.6.1 and 2.3.1]{Tu08}, but Theorem~\ref{thm:ClassicVSlice} gives infinitely 
many distinct virtual concordance classes for which $u_\pm(K)$ all vanish.
 
\begin{corollary} \label{cor:ClassicVSlice}
Two classical knots are virtually concordant if and only if they are concordant as classical knots.
\end{corollary}

\begin{proof}
Given two classical knots $K_0$ and $K_1$, apply the Theorem~\ref{thm:ClassicVSlice} 
to the connected sum~$K_0 \# -K_1$, where
$-K_1$ denotes the mirror image of $K_1$ with its orientation reversed.
\end{proof}

Suppose the classical knot~$K \subset S^2 \times I$ is virtually slice, then we can
find a $3$-manifold~$W$ which is a filling of $S^2$ and a 
slice disk~$\Delta \subset W \times I$ cobounding the knot~$K$. To transfer the slice disk from
$W\times I$ into $D^4$, we construct an embedding of the universal cover~$\wt
W$ into $D^3$.  The universal cover~$\wt W$ will have boundary~$\partial \wt W$ consisting of
many copies of $S^2$. A \emph{compression} of $\wt W$ is 
a smooth embedding~$\phi \colon \wt W \to D^3$ 
which restricts to an orientation-preserving diffeomorphism $S \to S^2$ on one of 
the boundary components~$S$ of $\partial W$.

We will construct compressions of $W$ from compressions of the prime parts of $W$.
\begin{lemma}\label{lem:PrimeNeat}
Let $W$ be a connected, compact, oriented and prime $3$-manifold with boundary~$\partial W \cong S^2$. Then
its universal cover~$\wt W$ admits a compression.
\end{lemma}
\begin{proof}
We can fill the boundary component of $W$ with a $3$-ball~$B$ and 
obtain a closed $3$-manifold~$W \cup B$.
The universal cover of $W \cup B$ is diffeomorphic to one of the 
manifolds $S^3$, $S^2 \times \RR$ or $\RR^3$. If $W \cup B$ is a geometric piece 
in the sense of Thurston, this
can be deduced from geometrisation and checking each geometry~\cite[Section 5]{Sc83}.  
If $W$ contains an incompressible torus, its universal cover is
diffeomorphic to the space~$\RR^3$~\cite[Theorem 1]{HRS89}.

As $S^2 \times \RR$ and $\RR^3$ embed into $S^3$, we may assume that we have an embedding of 
$\wt {W \cup B}$ into $S^3$. By post-composing with a diffeomorphism,
we may assume that a lift $B'$  of $B$ is mapped to the standard $3$-ball $D^3 \subset S^3$.
Denote the boundary of $B'$ by $S$. We have the following chain of embeddings 
\[ \wt W \subset \wt {W \cup B} \sm B' \subset S^3 \sm B' = D^3, \]
which gives a compression of $\wt W \subset D^3$.
\end{proof}
\begin{lemma}\label{lem:NeatExists}
Let $W$ be a connected, oriented, compact $3$-manifold with boundary~$\partial W \cong S^2$. Then
its universal cover $\wt W$ admits a compression.
\end{lemma}
\begin{proof}
We fix a prime decomposition~$S:=\{S_i\}$ of the $3$-manifold~$W$. This
is a finite collection~$S$ of disjointly embedded separating $2$-spheres~$S_i$ such that $2$-surgery on these
spheres gives a $3$-manifold whose components $W_1, \ldots, W_k$ are all prime $3$-manifolds.

After relabeling, we may assume that $W_1$ has boundary~$\partial W_1 \cong S^2$.
Take $\pi\colon \wt W \to W$ to be a universal cover.
The components of the preimages~$\pi^{-1}(S_i)$ are again $2$-spheres, which form the
collection $\wt S$. The spheres~$C \in \wt S$ are again separating: each sphere~$C$
cuts~$\wt W$ into two half-spaces. 
Given an orientation~$\sigma$ on the sphere~$C$, there is a unique half-space
$C^\sigma$ whose boundary orientation on $C$ is $\sigma$. To any subset 
\[ I \subseteq \left\{ (C, \sigma) \mid C \in \wt S \text{ and } \sigma \text{ an orientation on } C \right\} \] 
we can associate the submanifold 
 ${\bigcap_{(C, \sigma) \in I}} C^\sigma$, which is
an intersection of half-spaces of $\wt W$.
We call a submanifold $B \subseteq \wt W$ \emph{chunked} if $B ={\bigcap_{(C, \sigma) \in I}} C^\sigma$ for a subset $I$.
If $B$ is chunked, then its boundary components are contained in $\wt S$ or in the boundary $\partial \wt W$ of $\wt W$ itself.
Note that if $I$ is empty, then $\bigcap_{I} C^\sigma=\wt W$, thus $\wt W$ is chunked.

Given a chunked submanifold~$B$ and a boundary sphere~$C \in \wt S$ of $B$, there
is a unique smallest chunked submanifold $B' \supset B$ such that $C$ is in the interior of $B'$.
It is of the form $B' = B \cup_C P$ for a universal cover $P$ of a prime $3$-manifold.
We call $B'$ an \emph{elementary extension} of $B$ along $C$.

Fix a boundary component~$T \subset \partial \wt W$. Consider the following set
\[ Z:= \left\{ \phi \colon B \to D^3 \mid T \subset B, \phi(T) = S^2, B \text{ is chunked, and $\phi$ is a compression} \right\}.\]
We give $Z$ the partial order of the poset of maps, i.e. for $\phi \colon B \to D^3$
and $\phi' \colon B' \to D^3$, we declare $\phi' \geq \phi$ if and only if
$B \subset B'$ and $\phi'$ restricts to $\phi$.

By Lemma \ref{lem:PrimeNeat}, the set $Z$ is non-empty. Also totally ordered chains have
a maximal element, so $Z$ has a maximal element. Let $\phi \colon B \to D^3$ be maximal. 
We claim $B = \wt W$, which proves the lemma.

Pick a boundary sphere~$C \in \wt S$ of $B$ and denote by $B' = B \cup_C P$ the elementary extension of $B$ along $C$.
We construct a compression of $B'$ restricting to $\phi$. Consider $\phi(C) \subset D^3$. It is 
a smoothly embedded $2$-sphere in $D^3$. 
It separates the $3$-ball~$D^3$ into two components: an annulus and another $3$-ball~$D'$. Consequently,
the interior of the $3$-ball~$D'$ is disjoint from the image of $\phi(B)$. By
Lemma~\ref{lem:PrimeNeat}, we can embed $\phi_P \colon P \to D'$. As $\Diff^+(S^2)$ 
is path-connected, we can make $\phi_P$ agree with $\phi$ on $C$ and thus we obtain a compression
\[ \phi \cup_C \phi_P \colon B' \to D^3 \]
extending $\phi$.
\end{proof}

Using the compression of Lemma~\ref{lem:NeatExists}, 
we show how to transfer a slice disk for a virtually slice classical knot to the 4-ball.
\begin{proof}[Proof of Theorem~\ref{thm:ClassicVSlice}]
Let $K$ be a classical knot which is virtually slice.
By definition, the knot~$K$ is embedded in a thickened $2$-sphere and
there is a filling~$W$ of $S^2$ together with a slice disk~$\Delta \subset W \times I$
cobounding the knot $K$ in the boundary~$\partial W \times I$.

Let $\wt W \to W$ be a universal cover
and $\phi \colon \wt W \to D^3$ be a compression which exists by Lemma~\ref{lem:NeatExists}.
Let $T \subset \partial \wt W$ be a boundary sphere which is mapped via $\phi$ to the boundary of $D^3$. 
The product map $\wt W \times I \to D^3 \times I$ is also covering map.
As the slice disk $\Delta$ is contractible, it lifts to a disk $\wt \Delta \subset \wt W \times I$
with boundary $\partial \wt \Delta \subset T \times I$. Note that $\partial \wt \Delta$ is
still the knot~$K$.

Now $\phi( \Delta ) \subset D^3 \times I \cong D^4$ is a slice disk for $K$.
\end{proof}

\section{The virtual knot concordance group}\label{section3}
In this section, we  introduce concordance of long virtual knots and the virtual 
knot concordance group~$\vConG$. 
We then apply Theorem~\ref{thm:ClassicVSlice} to deduce injectivity of the natural 
homomorphism~$\ConG \to \vConG$, where $\ConG$ is the classical concordance group.
\subsection{Long virtual knots}

The group operation in $\ConG$ and $\vConG$ is given by connected sum. For round virtual knots, this operation is not well-defined because it depends on the choice of diagram and  on where the diagrams are connected.
These ambiguities disappear if one instead works with long virtual knots.
  
Recall that a long virtual knot diagram is a regular immersion of $\RR$ in the plane
$\RR^2$ which is identical with the $x$-axis outside a compact set, which we will principally take to be the closed ball $B_0(R)$ of radius $R$ centered at the origin.
Double points of the immersion can occur only inside $B_0(R),$ and each one is
labelled either classical or virtual, indicated as before with an over- or
undercrossing if classical or by encircling the crossing if virtual. 
Two such diagrams are equivalent if one can be related to the other by
a compactly supported planar isotopy and a finite sequence of generalised
Reidemeister moves. A \emph{long virtual knot} $K$ is defined to be an equivalence class of long
virtual knot diagrams.
We call the long knot given by the $x$-axis the \emph{long unknot}.
Note that by convention, all long virtual knots are oriented from left to right.

The connected sum of two long virtual knots $K_0$ and $K_1$, denoted $K_0 \# K_1$, is defined by concatenation with $K_0$ on the left and $K_1$ on the right. 
It is easy to verify that long virtual knots form a monoid under connected sum with identity given by the long unknot.

\begin{remark}
The connected sum on long virtual knots is not a commutative operation~\cite[Theorem 9]{Ma08}.
\end{remark}

\subsection{The virtual knot concordance group}

We now extend the notion of virtual concordance to long virtual knots, following Kauffman \cite{Ka15}. 
\begin{definition}
\begin{enumerate}
\item Two long knots~$K_0$ and $K_1$ are \emph{virtually concordant} if one can
be obtained from the other by generalised Reidemeister moves and a finite
sequence of saddle moves, births, and deaths such that the number of saddle
moves equals the sum of births and deaths. 
\item A long virtual knot is \emph{virtually slice} if it is virtually concordant to the long unknot. 
\end{enumerate}
\end{definition}

We will use $[K]$ to denote the concordance class of a long virtual knot $K$ and  
\[ \vConG = \{[K] \mid \text{$K$ is a long virtual knot} \}\] 
for the set of concordance classes of long virtual knots. It is immediate from the definition
that the concordance class of the connected sum $K_0 \#K_1$ depends only on the
concordance classes of $K_0$ and $K_1$. This shows  that the operation of
connected sum descends to a well-defined operation on $\vConG$. 
Thus $(\vConG,\#)$ is a monoid. 

Turaev observes that $(\vConG,\#)$ is actually a group~\cite[Section 5.2]{Tu08}. 
Just as with classical knots, the inverse of $[K]$ is obtained by
taking the mirror image and reversing the orientation. Specifically, given a
long virtual knot $K$,  let $K^m$ be the long virtual knot obtained by
reflecting $K$ through the vertical line $x=R$, and let $-K$ be the
result of reversing the orientation of $K^m$.
Chrisman~\cite[Theorem 1]{Ch16} proves that $K \# -K$ is virtually
slice, and thus it follows that $[-K]$ is the inverse of $[K]$ in
$(\vConG,\#)$.

Given a long virtual knot $K$, let $\overline{K}$ denote its closure. Thus, $\overline{K}$ is the round virtual knot obtained by discarding the parts of $K$ outside the closed ball $B_0(R)$ and joining the points $(R,0)$ to $(-R,0)$ on $K$ with the semicircle $(R \cos(\th), -R \sin(\th)) \subset \RR^2$ for $0 \leq \th \leq \pi$.
 
\begin{lemma} \label{lem:virtual-slice-equivalence}
A long virtual knot $K$ is virtually slice if and only if its closure $\overline{K}$ is virtually slice.
\end{lemma}

\begin{proof}
Suppose $K$ is virtually slice. Then there is a finite sequence of births, deaths, and saddles and planar isotopies taking $K$ to the trivial long knot. We can choose $R$ sufficiently large so that all births, deaths, and saddles take place in the ball $B_0(R)$. Since planar isotopies are compactly supported, we can assume that $K$ is unchanged outside of $B_0(R)$.
Thus, the same set of births, deaths, and saddle moves and planar isotopies show that $\overline{K}$ is virtually concordant to the round unknot.

To see the other direction, represent $K$ as a long virtual knot diagram which
coincides with the $x$-axis outside the open ball $B_0(R)$. Construct a new
diagram for $K$ by translating the original diagram vertically and connecting the points $(-R,2R)$ and $(R,2R)$
on the new diagram to the $x$-axis using vertical lines. Now perform a saddle move
and replace the vertical line segments from
$(-R,0)$ to $(-R,R)$ and from $(R,R)$ to $(R,0)$ with the horizontal line
segments from $(-R,0)$ to $(R,0)$ and from $(R,R)$ to $(-R,R)$. This saddle
move transforms $K$ into a 2-component link with one component the trivial
long knot and the other component the round virtual knot $\overline{K}$, which by hypothesis 
bounds a slice disk $\De$. Capping $\overline{K}$ off with $\De$ gives a
virtual concordance from $K$ to the trivial long knot. It follows that $K$ is
virtually slice.
\end{proof}

Recall that for classical knots, the map $K \mapsto \overline{K}$ gives a one-to-one correspondence between long knots and round knots.
From the definition of virtual concordance, one deduces that 
the natural inclusion map from classical knots to virtual knots induces  
a well-defined homomorphism~$\psi \colon \ConG \to \vConG$.
The next result is then an immediate consequence of Theorem~\ref{thm:ClassicVSlice} 
and Lemma~\ref{lem:virtual-slice-equivalence}.
\begin{corollary}
The homomorphism $\psi\colon \ConG \to \vConG$ is injective.
\end{corollary}
 
It is an open question whether the concordance group $\vConG$ of long virtual
knots is abelian, see \cite[Section 6.5]{Tu08}. Another interesting open
problem is to determine the structure of $\vConG$, for instance can one
describe the cokernel of the map $\psi$? Does it contain torsion elements? 

Turaev introduces many useful invariants of virtual knot concordance
in~\cite{Tu08}. These include the polynomials $u_\pm(K)$ and the graded genus
$\si(T)$ of the graded matrix $T=(G,s,b)$ associated to a virtual knot $K$. Any
virtual knot $K$ with $u_+(K) \neq 0$ or $u_-(K)\neq 0$ will have infinite
order in $\vConG$ \cite[Proposition 2]{Ch16}. However, if
$K$ is classical, then these invariants vanish, and we view it as an
interesting challenge to derive new invariants of virtual knot concordance to shed light on these questions.

\begin{ackn}
We thank Stefan Friedl for bringing us together and for many fruitful discussions.
The authors would also like to thank Micah Chrisman, Isabel Gaudreau, and Mark Powell for their input and feedback. 

H. Boden is grateful to the University of Regensburg for its hospitality. He was supported by a grant from the Natural Sciences and Engineering Research Council of Canada.
M. Nagel thanks McMaster University for its hospitality. 
He was supported by a CIRGET postdoctoral fellowship
and by SFB 1085 at the University of Regensburg funded by the DFG.
 \end{ackn}


\bibliographystyle{alpha}

\end{document}